\documentclass[11pt]{article}
\usepackage[utf8]{inputenc}
\usepackage{a4wide}
\usepackage{amsmath}
\usepackage{amssymb}
\usepackage{amsthm}

\usepackage[dvips]{graphicx}

\newcommand{\dfn}[1]{\textit{#1}} 
\newcommand{\vd}[1]{deg(#1)} 

\newcommand{\inter}[1]{\operatorname{int}(#1)}
\newcommand{\ext}[1]{\operatorname{ext}(#1)}

\theoremstyle{plain} \newtheorem{thm}{Theorem}[section]
\theoremstyle{plain} \newtheorem{prop}[thm]{Proposition}
\theoremstyle{plain} \newtheorem{lem}[thm]{Lemma}
\theoremstyle{plain} \newtheorem{cor}[thm]{Corollary}
\theoremstyle{plain} 
\theoremstyle{plain} 
\theoremstyle{plain} 

\newenvironment{proofwithoutqed}{\emph{Proof.}}{\hfill\null}

\title{Geometric representations of binary codes embeddable in three dimensions}
\author{Pavel Rytíř\thanks{Supported by the Czech Science Foundation under the contract no. 201/09/H057 and by the grant \mbox{SVV-2010-261313} (Discrete Methods and Algorithms).}\\ {\small Department of Applied Mathematics,} {\small Charles University in Prague,}\\ {\small Malostranské~náměstí~25}, {\small Prague 118 00, Czech Republic}}

\begin{document}
\maketitle

\begin{abstract}
We say that a binary linear code $\mathcal{C}$ has a geometric representation if there exists a two dimensional simplicial complex $\Delta$ such that $\mathcal{C}$ is a punctured code of the kernel $\ker\Delta$ of the incidence matrix of $\Delta$ and $\dim\mathcal{C}=\dim\ker\Delta$. We show that every binary linear code has a geometric representation that can be embedded into $\mathbb{R}^4$. Moreover, we show that a binary linear code $\mathcal{C}$ has a geometric representation in $\mathbb{R}^3$ if and only if there exists a graph $G$ such that $\mathcal{C}$ equals the cut space of $G$. This is a polynomially testable property and hence we can conclude that there is a polynomial algorithm that decides the minimal dimension of a geometric representation of a binary linear code.
\end{abstract}

\section{Introduction}
This paper extends results of Rytíř~\cite{rytir2,rytir3} where it was proven that every binary linear code has a geometric representation.
Here we show that each binary linear code has a geometric representation that can be embedded into $\mathbb{R}^4$. Moreover we characterize those $\mathcal{C}$ which admit a geometric representation in $\mathbb{R}^3$.

A \dfn{linear code $\mathcal{C}$ of length $n$ and dimension $d$ over a field $\mathbb{F}$} is a linear subspace with dimension $d$ of the vector space $\mathbb{F}^n$. Each vector in $\mathcal{C}$ is called a \dfn{codeword}. Let $B$ be a basis of a binary code $\mathcal{C}$. A basis $B$ is k-basis if every entry is non-zero in at most $k$ vectors of $B$.


Let $\mathcal{C}\subseteq\mathbb{F}^n$ be a linear code over a field $\mathbb{F}$ and let $S$ be a subset of $\left\lbrace 1,\dots,n\right\rbrace $. \dfn{Puncturing} a code $\mathcal{C}$ along $S$ means deleting the entries indexed by the elements of $S$ from each codeword of $\mathcal{C}$. The resulting code is denoted by $\mathcal{C}/S$.

A \dfn{simplex} $X$ is the convex hull of an affine independent set $V$ in $\mathbb{R}^d$. The \dfn{dimension} of $X$ is $\left\lvert V\right\rvert-1$, denoted by $\dim X$.
The convex hull of any non-empty subset of $V$ that defines a simplex is called a \dfn{face} of the simplex.
A \dfn{simplicial complex} $\Delta$ is a set of simplices fulfilling the following conditions: Every face of a simplex from $\Delta$ belongs to $\Delta$ and the intersection of every two simplices of $\Delta$ is a face of both.

The dimension of $\Delta$ is $\max\left\lbrace\dim X\vert X\in\Delta\right\rbrace$.
Let $\Delta$ be a $d$-dimensional simplicial complex. We define the \dfn{incidence matrix} $A=\left(A_{ij}\right)$ as follows: The rows are indexed by $\left(d-1\right)$-dimensional simplices and the columns by $d$-dimensional simplices. We set
\begin{equation*}
A_{ij}:=\begin{cases}
         1& \text{if }(d-1)\text{-simplex }i\text{ belongs to }d\text{-simplex }j,\\
         0& \text{otherwise}.
        \end{cases}
\end{equation*}
This paper studies two dimensional simplicial complexes where each maximal simplex is a triangle or a segment. We call them \dfn{triangular configurations}.
Let $\Delta$ be a triangular configuration. A \dfn{subconfiguration} of $\Delta$ is a subset of $\Delta$ that is a triangular configuration.
We denote the set of triangles of $\Delta$ by $T(\Delta)$.
The \dfn{cycle space} of $\Delta$ over a field $\mathbb{F}$, denoted $\ker\Delta$, is the kernel of the incidence matrix $A$ of $\Delta$ over $\mathbb{F}$, that is
$\{x\vert Ax=0\}$.
Let $T$ be a subset of the set of triangles of $\Delta$.
We denote by $\mathcal{K}(T)$ the triangular configuration that is defined by the set of triangles $T$.
The \dfn{even subset} or \dfn{cycle} of $\Delta$ is a subset $E$ of the set of triangles of $\Delta$ such that all edges of the triangular configuration $\mathcal{K}(E)$ have an even degree.

Let $\{t_1,\dots,t_m\}$ be the set of triangles of $\Delta$.
For a subconfiguration $\Delta'$ of $\Delta$, we let $\chi(\Delta')=(\chi(\Delta')_1,\dots,\chi(\Delta')_m)\in\left\lbrace 0,1\right\rbrace^m$ denote its \dfn{characteristics vector}, where $\chi(\Delta')_i=1$ if $\Delta'$ contains triangle $t_i$, and $\chi(\Delta')_i=0$ otherwise.
Note that, the characteristics vectors of even subsets of $\Delta$ forms the cycle space of $\Delta$.

Let $E_1$ and $E_2$ be sets.
Then the \dfn{symmetric difference} of $E_1$ and $E_2$, denoted by $E_1\bigtriangleup E_2$, is defined to be $E_1\bigtriangleup E_2:=\left(E_1\cup E_2\right)\setminus\left(E_1\cap E_2\right)$. Note that, the symmetric difference of two even subsets $E_1$ and $E_2$ of $\Delta$ is also even subset of $\Delta$ and it holds $\chi(\mathcal{K}(E_1))+\chi(\mathcal{K}(E_2))=\chi(\mathcal{K}(E_1\bigtriangleup E_2))$ over $GF(2)$.


A linear code $\mathcal{C}$ has a \dfn{geometric representation} if there exists a triangular configuration $\Delta$ such that $\mathcal{C}=\ker\Delta/S$ for some set $S$ and $\dim\mathcal{C}=\dim\ker\Delta$. For such $S$ we write $S=S(\ker\Delta,\mathcal{C})$.
\begin{thm}[Rytíř \cite{rytir3}]
\label{thm:repr1}
Let $\mathcal{C}$ be a linear code over rationals or over $GF(p)$, where $p$ is a prime. Then $\mathcal{C}$ has a geometric representation.
\end{thm}

\subsection{Main Results}
A basis $B$ of a binary linear code $\mathcal{C}\subseteq GF(2)^n$ is 2-basis if every entry $i\leq n$ is non-zero in at most two vectors of $B$.
\begin{thm}
\label{thm:embhas2bas}
Let $\Delta$ be a triangular configuration embeddable into $R^3$ then $\ker\Delta$ has a 2-basis.
\end{thm}
\begin{proof}
The proof follows from Theorem~\ref{thm:s2basis} in Section~\ref{sec:proofembhas2bas}.
\end{proof}

By Whitney's theorem, the cycle space of a 3-connected graph $G$ determines $G$. It is therefore natural to ask whether our result can help to answer the question: Given a 2 dimensional simplicial complex, is it embeddable into $\mathbb{R}^3$? Theorem~\ref{thm:embhas2bas} gives only a necessary condition. For example no triangulation of the Klein bottle can be embedded into $\mathbb{R}^3$ and its cycle space has a 2-basis.
The topic of embedding of simplicial complexes is treated in Matoušek~et~al.~\cite{matousek1}.


The main result of this paper is that existence of a 2-basis characterize geometric representations in $\mathbb{R}^3$.

\begin{thm}
\label{thm:repr3d}
A binary linear code $\mathcal{C}$ has a geometric representation embeddable into $\mathbb{R}^3$ if and only if $\mathcal{C}$ has a 2-basis.
\end{thm}
The above theorem is an analogy of Mac Lane's planarity criterion~\cite{maclane} for graphs.

\begin{thm}
\label{thm:repr3dgraphs}
A binary linear code $\mathcal{C}$ has a geometric representation embeddable into $\mathbb{R}^3$ if and only if there exists a graph $G$ such that $\mathcal{C}$ equals the cut space of $G$.
\end{thm}

It is well known that every two dimensional simplicial complex can be embedded into $\mathbb{R}^5$. Hence, every binary linear code has a geometric representation embeddable into $\mathbb{R}^5$. We further show:
\begin{thm}
\label{thm:repr4d}
Every binary linear code $\mathcal{C}$ has a geometric representation embeddable into $\mathbb{R}^4$.
\end{thm}

Theorem~\ref{thm:repr4d} extends a main result of Rytíř~\cite{rytir2} where it is shown that every binary linear code has a geometric representation. 
\begin{cor}
\label{cor:dec3d}
There is a polynomial algorithm that decides the minimal dimension of a geometric representation of a binary code $\mathcal{C}$.
\end{cor}

This positive result complements the results of Matoušek~et~al.~\cite{matousek1} on embeddings of simplicial complexes.

\section{Proof of main results}
\begin{proof}[Proof of Theorem~\ref{thm:repr3d}]
The necessary condition of the theorem follows from Theorem~\ref{thm:embhas2bas}. The sufficient condition is proven in Section~\ref{sec:repr3d}.
\end{proof}

\subsection{Bases of triangular configurations embedded into $\mathbb{R}^3$}
\label{sec:embhas2bas}
In this section we suppose that all triangular configurations are embedded into $\mathbb{R}^3$ with the standard Euclidean metric $\rho(x,y):=\sqrt{\sum_{i=1}^3(x_i-y_i)^2}$. Let $x$ be an element of $\mathbb{R}^3$ and let $\epsilon\in\mathbb{R}$ and $\epsilon>0$. The $\epsilon$-neighborhood of $x$ is the set $N_\epsilon(x):=\{y\in\mathbb{R}^3\,\vert\, \rho(x,y)<\epsilon\}$.
If no confusion can arise we let $N_\epsilon(x)=N(x)$.
Let $(x_1,\dots,x_m)$ be a sequence points in a space. A \dfn{polygonal path} along the sequence $(x_1,\dots,x_m)$ is a sequence of line segments connecting the consecutive points.
Let $\Delta$ be a triangular configuration embedded into $\mathbb{R}^3$.
A \dfn{cell} $X$ of $\Delta$ is a non-empty maximal subset of $\mathbb{R}^3\setminus\Delta$ with respect to inclusion such that between any two points of $X$ there is a polygonal path that does not intersect $\Delta$.
A \dfn{bounded cell} is a cell that is contained in some sphere of a finite diameter.
A \dfn{strong boundary} is a triangular configuration $C$ such that $C$ has at least two cells and every subconfiguration $C'$ has fewer cells than $C$. An example is depicted in Figure~\ref{fig:strongboundary}.
\begin{figure}
 \centering
 \includegraphics{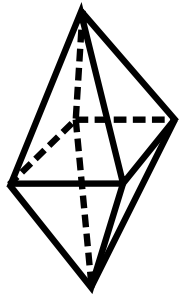}
 \caption{An example of strong boundary in $\mathbb{R}^3$}
 \label{fig:strongboundary}
\end{figure}
A one dimensional counterpart of strong boundary is a polygon, for example see Figure~\ref{fig:polygon}.
\begin{figure}
 \centering
 \includegraphics{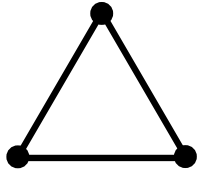}
 \caption{A polygon, counterpart of strong boundary in $\mathbb{R}^2$}
 \label{fig:polygon}
\end{figure}
Let $X$ be a subset of $\mathbb{R}^3$. The closure of $X$, denoted by $cl(X)$, is the set $cl(X):=\{y\in\mathbb{R}^3\,\vert\, \forall\epsilon>0 ; N_\epsilon(y)\cap X\neq\emptyset\}$.
We say that a triangle $t$ is \dfn{incident} with a cell $S$ if $t\subseteq cl(S)$.

\begin{prop}
\label{prop:trianglecells}
\label{prop:twocells}
Let $\Delta$ be a triangular configuration embedded into $\mathbb{R}^3$. Then every triangle $t$ of $\Delta$ is incident with at least one cell of $\Delta$ and at most two cells of $\Delta$.
\end{prop}
\begin{proof}
Let $t$ be a triangle of $\Delta$. For a contradiction, suppose that $t$ is incident with three cells $X_1,X_2,X_3$ of $\Delta$.
Let $p$ a point of $t$ that does not belong to any edge of $t$.
It holds that $p\in cl(X_1)$, $p\in cl(X_2)$ and $p\in cl(X_3)$.
Let $N(p)$ be a neighborhood of $p$ such that $N(p)$ does not intersect an edge of $\Delta$. The neighborhood $N(p)$ intersects the cells $X_1,X_2,X_3$. Let $x_1,x_2,x_3$ be points of $cl(X_1)\cap N(p)$, $cl(X_2)\cap N(p)$, $cl(X_3)\cap N(p)$, respectively. Then, the segments $x_1x_2$, $x_2x_3$, $x_1x_3$ intersect triangle $t$. Let $H$ be a hyperplane of $\mathbb{R}^3$ that contains triangle $t$. Then two points of $x_1,x_2,x_3$ belong to the same half-space defined by $H$.
The segment connecting these two points do not intersect $t$. This is the contraction.
Hence, $t$ is incident with at most two cells of $\Delta$.

Now, we show that $t$ is incident with at least one cell.
Let $v_1,v_2,v_3$ be vertices of $t$ and let $p$ be a point of $t$ that belongs to no edge of $\Delta$.
Let $v$ be a vector orthogonal to triangle $t$ and let $\epsilon>0$.
Let $P^+_\epsilon$ be a convex hull of set $\{v_1,v_2,v_3,p+\epsilon v\}$ and let $P^-_\epsilon$ be a convex hull of set $\{v_1,v_2,v_3,p-\epsilon v\}$.
We choose $\epsilon>0$ sufficiently small such that $\Delta\cap P^+_\epsilon=t$ and $\Delta\cap P^-_\epsilon=t$ .
The sets $P^+_\epsilon\setminus t$ and $P^-_\epsilon\setminus t$ are convex and disjoint with $\Delta$. Thus, $P^+_\epsilon\setminus t$ is a part of one cell of $\Delta$. Let $X^+$ be the cell of $\Delta$ that contains $P^+_\epsilon\setminus t$. Clearly $t\subseteq cl(X^+)$. Thus, triangle $t$ is incident with at least cell $X^+$.


\end{proof}

\begin{cor}
\label{cor:boundtricell}
Let $C$ be a strong boundary embedded into $\mathbb{R}^3$. Then every triangle $t$ of $C$ is incident with two cells of $C$.
\end{cor}
\begin{proof}
By Proposition~\ref{prop:twocells}, triangle $t$ is incident with one or two cells of $C$.
If $t$ is incident with one cell, we can remove it from $C$ and the number of cells of $C$ does not change. Thus, $C\setminus\{t\}$ is also a strong boundary. This contradict with the minimality of $C$. Hence, $t$ is incident with exactly two cells. 
\end{proof}

\begin{lem}
\label{lem:cellsjoin}
Let $\Delta$ be a triangular configuration embedded into $\mathbb{R}^3$. Let $t$ be a triangle of $\Delta$ incident with two cells of $\Delta$. Then the number of cells of $\Delta\setminus\{t\}$ is equal to the number of cells of $\Delta$ minus one.
\end{lem}
\begin{proof}
Let $X_1$ and $X_2$ be cells incident with $t$.
Let $x$ be a point of $t$.
Then there are points $x_1$ and $x_2$ of $X_1$ and $X_2$, respectively, such that $N(x_1)\cap x\neq\emptyset$ and $N(x_2)\cap x\neq\emptyset$. Hence, there is a polygonal path between $x_1$ and $x_2$ disjoint from $\Delta\setminus\{t\}$. The set $X_1\cup t\cup X_2$ is a cell of $\Delta\setminus\{t\}$ and the proposition follows.
\end{proof}

\begin{prop}
\label{prop:div2cells}
Let $C$ be a strong boundary embedded into $\mathbb{R}^3$.
Then $C$ has exactly two cells.
\end{prop}
\begin{proof}
For a contradiction suppose that $C$ has more than two cells.
Let $t$ be a triangle of $C$. By Corollary~\ref{cor:boundtricell}, triangle $t$ is incident with exactly two cells. By Lemma~\ref{lem:cellsjoin}, by removing $t$ from $C$, we join two cells into one.
If $C$ has more than two cells, the subconfiguration $C\setminus\{t\}$ has at least two cells.
Let $C'$ be the minimal subconfiguration (with respect to inclusion) of $C\setminus\{t\}$ that has at least two cells. Then $C'$ is a smaller strong boundary than $C$. This is a contradiction with the minimality of $C$.
\end{proof}

\begin{prop}
Let $C$ be a strong boundary embedded into $\mathbb{R}^3$.
Then one of the cells of $C$ is bounded and the second one is unbounded.
\end{prop}
\begin{proof}
By proposition~\ref{prop:div2cells}, $C$ has two cells.
By definition, every triangular configuration $\Delta$ is finite. Thus, every strong boundary $C$ is finite.
Hence, $C$ is contained in a sufficiently large sphere $S$.
The complement of the ball of $S$ is contained in one cell of $C$, thus this cell is unbounded. The other cell of $C$ is inside this ball and thus it is bounded.

\end{proof}
Let $C$ be a strong boundary. We call the bounded cell of $C$ \dfn{inner cell} of $C$ and denote it by $\inter{C}$. The unbounded cell of $C$ we denote by $\ext{C}$. We denote $C\cup\inter{C}$ and $C\cup\ext{C}$ by $\overline{\inter{C}}$ and $\overline{\ext{C}}$, respectively.
So far we considered strong boundary as a triangular configuration in $\mathbb{R}^3$. Now we consider strong boundaries in a triangular configuration $\Delta$. We say that a strong boundary $C$ is a \dfn{strong boundary of $\Delta$} if $C$ is a subconfiguration of $\Delta$.
We say that a triangular configuration $\Delta$ is connected if every two triangles of $\Delta$ belong to a common strong boundary of $\Delta$.
The \dfn{connected component} of $\Delta$ is a maximal connected subconfiguration (under inclusion) of $\Delta$.

\begin{prop}
\label{prop:closureboundary}
Let $C$ be a strong boundary.
Then $cl(\inter{C})=\overline{\inter{C}}$ and $cl(\ext{C})=\overline{\ext{C}}$.
\end{prop}
\begin{proof}
By Corollary~\ref{cor:boundtricell}, every triangle of $C$ is incident exactly with two cells $\inter{C}$ and $\ext{C}$. By definition of incidence, it holds $cl(\inter{C})=\overline{\inter{C}}$ and $cl(\ext{C})=\overline{\ext{C}}$.
\end{proof}

\begin{figure}
 \centering
 \includegraphics{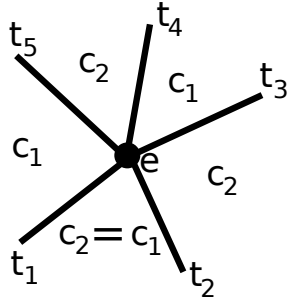}
\caption{A contradicting example of an edge $e$ of a strong boundary. The boundary has two cells $c_1,c_2$. The edge $e$ is incident with triangles $t_1,\dots,t_5$ and it has odd Degree 5. Then the strong boundary has only one cell. This contradict the definition of strong boundary.}
 \label{fig:odddegree}
\end{figure}

\begin{prop}
\label{prop:boundeven}
Let $C$ be a strong boundary embedded into $\mathbb{R}^3$.
Then the set of triangles of $C$ is an even subset.
\end{prop}
\begin{proof}
For a contradiction suppose that a strong boundary $C$ contains an edge $e$ with an odd degree.
By Proposition~\ref{prop:div2cells}, $C$ has two cells.
By Corollary~\ref{cor:boundtricell}, every triangle of $C$ is incident with two cells.
Let $T$ be the set of triangles incident with $e$.
Since edge $e$ has an odd degree and every triangle of $T$ is incident with two cells,
we set a contradiction.
(see Figure~\ref{fig:odddegree}).
\end{proof}

\subsubsection{Elementary strong boundaries}
A strong boundary $C$ of $\Delta$ is \dfn{elementary} if there is no strong boundary $C'$ of $\Delta$ such that $\inter{C}\cap\inter{C'}\neq\emptyset$ and $\inter{C}\cap\ext{C'}\neq\emptyset$.
First, we illustrate this definition on one dimensional simplicial complexes embedded into $\mathbb{R}^2$. One dimensional simplicial complexes embedded into $\mathbb{R}^2$ correspond to planar embeddings of planar graphs. The graphs counterpart of our definition of elementary strong boundary is a boundary of a face of a 2-connected plane graph. The 2-connected plane graph depicted in Figure~\ref{fig:graphelementarycircuits} has two boundaries of faces (elementary strong boundaries) depicted in Figure~\ref{fig:graphelementarycircuits2} and one circuit (strong boundary) that is not a boundary of a face (elementary strong boundary) depicted in Figure~\ref{fig:graphelementarycircuits3}.

\begin{figure}
 \centering
 \includegraphics{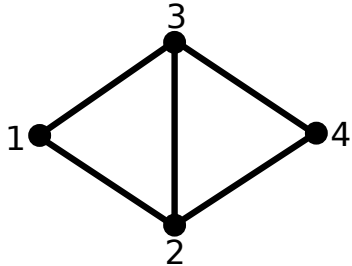}
 \caption{One dimensional complex embedded into $\mathbb{R}^2$ (a plane graph).}
 \label{fig:graphelementarycircuits}
\end{figure}
\begin{figure}
 \centering
 \includegraphics{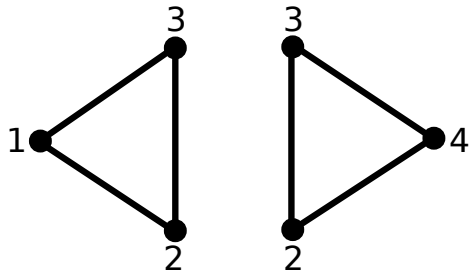}
 \caption{Two elementary strong boundaries of the complex in Figure~\ref{fig:graphelementarycircuits}.}
 \label{fig:graphelementarycircuits2}
\end{figure}
\begin{figure}
 \centering
 \includegraphics{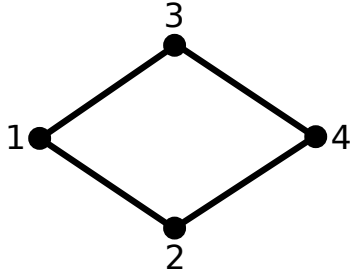}
 \caption{This strong boundary of the complex in Figure~\ref{fig:graphelementarycircuits} is not elementary.}
 \label{fig:graphelementarycircuits3}
\end{figure}

Now, we give example of triangular configuration embedded into $\mathbb{R}^3$ with two elementary strong boundaries. The triangular configuration in Figure~\ref{fig:elementaryboundary} has two elementary strong boundaries (Figure~\ref{fig:elementaryboundary2}) and one strong boundary that is not elementary (Figure~\ref{fig:elementaryboundary3}).

\begin{figure}
 \centering
 \includegraphics{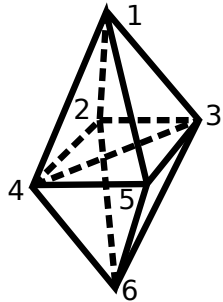}
 \caption{Triangular configuration embedded into $\mathbb{R}^3$.}
 \label{fig:elementaryboundary}
\end{figure}

\begin{figure}
 \centering
 \includegraphics{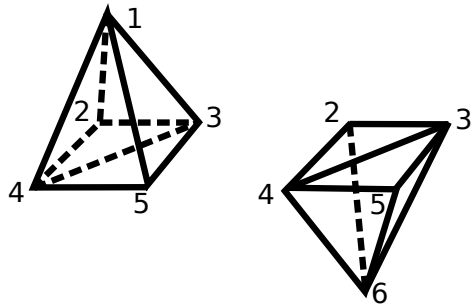}
 \caption{Two elementary strong boundaries of the triangular configuration in Figure~\ref{fig:elementaryboundary}.}
 \label{fig:elementaryboundary2}
\end{figure}

\begin{figure}
 \centering
 \includegraphics{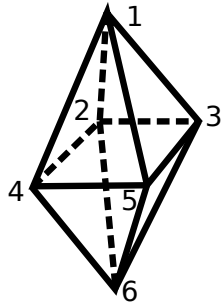}
 \caption{This strong boundary of the triangular configuration in Figure~\ref{fig:elementaryboundary} is not elementary.}
 \label{fig:elementaryboundary3}
\end{figure}

\begin{lem}
\label{lem:cellelembound}
Let $\Delta$ be a connected triangular configuration embedded into $\mathbb{R}^3$. Let $X$ be a bounded cell of $\Delta$. Let $E$ be the set of the triangles of $\Delta$ incident with $X$.
Then $\mathcal{K}(E)$ (triangular configuration defined by the set of triangles $E$) is an elementary strong boundary of $\Delta$.
\end{lem}
\begin{proof}
Since cell $X$ is bounded, set $E$ is nonempty.
Triangular configuration $\mathcal{K}(E)$ has at least two cells: If $\mathcal{K}(E)$ has only one cell, there is a triangle of $E$ that belongs to no strong boundary of $\Delta$. This contradict the connectivity of $\Delta$.

Now we show that triangular configuration $\mathcal{K}(E)$ has at most two cells: For a contradiction suppose that $\mathcal{K}(E)$ has at least three cells $X_1,X_2,X_3$.
Let $t_1$ be a triangle of $E$ incident with $X_1$ and $X_2$ and let $t_2$ be a triangle of $E$ incident with $X_2,X_3$.
Since $\Delta$ is connected, there is a strong boundary $D$ that contains $t_1$ and $t_2$.
The cell $X$ is a subset of $X_2$ and the cells $X,X_2$ are subsets of $\inter{D}$. Then $E\subseteq\overline{\inter{D}}$ and $X_1,X_3\subseteq\ext{D}$.
Hence, the triangular configuration $\mathcal{K}(E)$ does not have cells $X_1,X_3$, the contradiction.

Let $t$ be a triangle of $E$. We show that triangle $t$ is incident with two cells $X_1,X_2$ of $\mathcal{K}(E)$.
For a contradiction suppose that $t$ is incident only with cell $X_1$. It holds $X\subseteq X_1$. Then $t$ is incident only with cell $X$ of $\Delta$.
By connectivity, triangle $t$ belongs to a strong boundary $D$ of $\Delta$. By Proposition~\ref{prop:div2cells}, triangle $t$ is incident with two cells of $D$.
Since $D$ is a subconfiguration of $\Delta$, triangle $t$ is incident with two cells $X^1_\Delta,X^2_\Delta$ of $\Delta$. This is the contradiction.

By Lemma~\ref{lem:cellsjoin}, $\mathcal{K}(E)\setminus\{t\}$ has only one cell.
Hence $\mathcal{K}(E)$ is a strong boundary.

For a contradiction suppose that $\mathcal{K}(E)$ is not an elementary strong boundary of $\Delta$. Then there is a strong boundary $C$ of $\Delta$ such that $\inter{\mathcal{K}(E)}\cap\inter{C}\neq\emptyset$ and $\inter{\mathcal{K}(E)}\cap\ext{C}\neq\emptyset$.
Then there is a triangle $t'$ of $C$ that belongs to $\inter{\mathcal{K}(E)}$. If there is not such a triangle $t'$, then $\inter{\mathcal{K}(E)}\subseteq \inter{C}$ and $\inter{\mathcal{K}(E)}\subseteq \ext{C}$. This contradict that $\ext{C}$ and $\inter{C}$ are disjoint.
Since the cell $X$ is a subset of $\inter{\mathcal{K}(E)}$, triangle $t'$ also belongs to $E$.
Since $t'\in\inter{\mathcal{K}(E)}$, triangle $t'$ is incident only with cell $\inter{\mathcal{K}(E)}$ of $\mathcal{K}(E)$. This is the contradiction. Thus, $\mathcal{K}(E)$ is an elementary strong boundary of $\Delta$.
\end{proof}

\begin{lem}
\label{lem:elembound}
Let $\Delta$ be a connected triangular configuration embedded into $\mathbb{R}^3$. Let $t$ be a triangle of $\Delta$ that belongs to a strong boundary $C$ of $\Delta$. Then $t$ belongs to exactly one elementary strong boundary $C'$ of $\Delta$ such that $C'\subseteq\overline{\inter{C}}$.
\end{lem}
\begin{proof}
Let $X$ be a cell of $\Delta$ such that $X$ is incident with $t$ and $X\subseteq\inter{C}$. Cell $X$ is bounded.
Let $C'$ be the set of triangles of $\Delta$ incident with $X$.
By lemma~\ref{lem:cellelembound}, the triangular configuration $\mathcal{K}(C')$ is an elementary strong boundary of $\Delta$.
By Proposition~\ref{prop:closureboundary} and from $X\subseteq\inter{C}$, we have $C'\subseteq\overline{\inter{C}}$.

If there is an elementary strong boundary $C''$ of $\Delta$ different from $C'$ that contains $t$ such that $C''\subseteq\overline{\inter{C}}$.
We have $\inter{C'}\cap\inter{C''}\neq\emptyset$.
Since $C'\neq C''$, we have $\ext{C'}\cap\inter{C''}\neq\emptyset$ or $\ext{C''}\cap\inter{C'}\neq\emptyset$. This a contradiction with the definition of elementary strong boundary. Thus, $t$ is contained only in one elementary strong boundary that is contained in $\overline{\inter{C}}$.
\end{proof}

\begin{lem}
\label{lem:trianinc}
Let $\Delta$ be a connected triangular configuration embedded into $\mathbb{R}^3$. Let $t$ be a triangle of $\Delta$. Then $t$ is incident with two cells of $\Delta$.
\end{lem}
\begin{proof}
Since $\Delta$ is connected, there is a strong boundary $C$ of $\Delta$ that contains $t$.
By Corollary~\ref{cor:boundtricell}, $t$ is incident with two cells of $C$. Since $C$ is a subconfiguration of $\Delta$, $t$ is also incident with two cells of $\Delta$.
\end{proof}

\begin{lem}
\label{lem:elembound2}
Let $\Delta$ be a connected triangular configuration embedded into $\mathbb{R}^3$. Let $t$ be a triangle of $\Delta$ such that $t$ is contained in $\inter{C}$ where $C$ is a strong boundary of $\Delta$. Then $t$ belongs to exactly two elementary strong boundaries $C_1,C_2$ of $\Delta$ and $C_1\subseteq\overline{\inter{C}}$ and $C_2\subseteq\overline{\inter{C}}$.
\end{lem}
\begin{proof}
By lemma~\ref{lem:trianinc}, triangle $t$ is incident with two cells $X_1$ and $X_2$ of $\Delta$.
Let $C_1$ and $C_2$ be the sets of triangles incident with $X_1$ and $X_2$, respectively. By lemma~\ref{lem:cellelembound}, the sets $C_1$ and $C_2$ are elementary strong boundaries of $\Delta$.

Since $t\in\inter{C}$, cells $X_1$ and $X_2$ are subsets of $\inter{C}$. Thus, $C_1\subseteq\overline{\inter{C}}$ and $C_2\subseteq\overline{\inter{C}}$.

For a contradiction suppose that there is a third elementary strong boundary $C_3$ of $\Delta$ that contains $t$. Since $t$ is incident with two cells, we have $\inter{C_3}\cap\inter{C_1}\neq\emptyset$ or $\inter{C_3}\cap\inter{C_2}\neq\emptyset$. Without loose of generality we can suppose $\inter{C_3}\cap\inter{C_1}\neq\emptyset$. Since $C_3\neq C_1$, we have $\ext{C_1}\cap\inter{C_3}\neq\emptyset$ or $\ext{C_3}\cap\inter{C_1}\neq\emptyset$. This a contradiction with the definition of elementary strong boundary.
Thus, $t$ is contained in exactly two elementary strong boundaries of $\Delta$.
\end{proof}

\begin{prop}
\label{prop:boundcomb}
Let $\Delta$ be a connected triangular configuration embedded into $\mathbb{R}^3$ and let $C$ be a strong boundary of $\Delta$ and let $ESB(C)$ be the set of elementary strong boundaries of $\Delta$ contained in $\overline{\inter{C}}$. Then $\chi(C)$ equals the sum of the characteristics vectors of the elements of $ESB(C)$ over $GF(2)$. Thus, $\chi(C)=\sum_{S\in ESB(C)}\chi(S)$.
\end{prop}
\begin{proof}
Each element of $ESB(C)$ is contained in $\overline{\inter{C}}$.
Therefore,
$$\bigtriangleup_{S\in ESB(C)}T(S)\subseteq\overline{\inter{C}},$$
where $T(S)$ denotes the set of triangles of $S$.
Let $t$ be a triangle of $\Delta$ such that $t\subseteq\inter{C}$.
By Lemma~\ref{lem:elembound2}, $t$ is incident with two elementary boundaries $C_1$ and $C_2$ such that $C_1,C_2\in ESB(C)$.
Therefore,
$$\bigtriangleup_{S\in ESB(C)}T(S)\subseteq T(C).$$

Let $t$ be a triangle of $C$. By Lemma~\ref{lem:elembound}, $t$ belongs to exactly one elementary strong boundary from $ESB(C)$.
Therefore,
$$\bigtriangleup_{S\in ESB(C)}T(S)\supseteq T(C).$$
Hence,
$$\bigtriangleup_{S\in ESB(C)}T(S)= T(C)$$
and
$$\sum_{S\in ESB(C)}\chi(S)=\chi(C)$$
over $GF(2)$.
\end{proof}

\subsubsection{Non-empty even subsets divide $\mathbb{R}^3$}
\begin{prop}
\label{prop:divides}
Let $\Delta$ be a non-empty triangular configuration embedded into $\mathbb{R}^3$ with all edges of an even degree. Then $\Delta$ has at least two cells.
\end{prop}
\begin{proofwithoutqed}
This proof is a variation of a proof of Jordan curve theorem for polygonal paths that can be found in Courant~et~al.~\cite{wim}.
First, we introduce some notation. Let $t$ be a triangle. We denote by $\mathring{t}$ the interior of $t$, i.e., $\mathring{t}:=t\setminus(e_1\cup e_2\cup e_3)$ where $e_1,e_2,e_3$ are the edges of $t$. Let $e$ be an edge. We denote by $\mathring{e}$ the interior of $e$, i.e., $\mathring{e}:=e\setminus(v_1\cup v_2)$ where $v_1,v_2$ are the vertices of $e$.

Let $r$ be a vector in $\mathbb{R}^3$ that is neither parallel with a triangle nor an edge of $\Delta$. Let $x$ be a point of $\mathbb{R}^3\setminus\Delta$. Let $R(x)$ be the ray from $x$ in direction $r$. 
Suppose that $R(x)$ does not intersect any vertex of $\Delta$.
We define the following quantities:
Let $I_T(R(x),\Delta)$ denote the number of intersection of $R(x)$ with interiors of triangles of $\Delta$.
Let $e$ be an edge of $\Delta$ that is intersected by $R(x)$. Let $H$ be the plane defined by the edge $e$ and the ray $R(x)$. Let $n$ be the number of triangles incident with $e$ on one side of $H$ and $m$ number of triangles on the other side of $H$. Then we define $I_e(R(x),\Delta)$ as the minimum of $n$ and $m$.
Let $I_E(R(x),\Delta)$ be the sum of $I_e(R(x),\Delta)$ over all edges of $\Delta$ that are intersected by $R(x)$ on interiors.


We define the sum $I(R(x),\Delta):=I_T(R(x),\Delta)+I_E(R(x),\Delta)$ and the parity of $x$ as $P(R(x),\Delta):=I(R(x),\Delta)\bmod 2$.

Let $P$ be a polygonal path in $\mathbb{R}^3\setminus\Delta$.
We show that all points of $P$ have the same parity.
First, we prove the following lemma.
\end{proofwithoutqed}

\begin{lem}
Let $x$ and $x'$ be points in $\mathbb{R}^3$ such that
\begin{enumerate}
 \item the segment $xx'$ does not intersect $\Delta$,
  \item $R(x')$ intersect at least one edge,
 \item $R(x')$ does not intersect a vertex,
 \item $R(y)$ does not intersect an edge for $y\in xx'\setminus x'$.
\end{enumerate}
Then $P(R(x),\Delta)=P(R(y),\Delta)$ for all $y\in xx'$.
\end{lem}
\begin{proof}
All points of $xx'$ except $x'$ have the same parity, since the parity can only change when the ray hits or leave an edge. A nontrivial case is to show that $x$ and $x'$ have the same parity.
Let $E(R(x'))$ be the set of edges of $\Delta$ that are intersected by $R(x')$.
Let $H(R,e)$ be the hyperplane defined by $R$ and $e$. Let $n_e$ be the number of triangles incident with $e$ on the same side of $H_e$ as $x$ and let $m_e$ denote the number of triangles incident with $e$ on the other side of $H_e$.

By definition,  
$$I(R(x),\Delta)=I_T(R(x),\Delta)+I_E(R(x),\Delta),$$
and
\begin{equation}
\begin{split}
I(R(x'),\Delta)=&I_T(R(x),\Delta)-\sum_{e\in E(R(x'))}n_e+
I_E(R(x),\Delta)+\sum_{e\in E(R(x'))}\min\{n_e,m_e\}
\end{split}
\end{equation}
Since every edge $e$ of $\Delta$ has an even degree, $n_e+m_e$ is even. Hence, $n_e\equiv m_e\bmod 2$. Therefore $I(R(x),\Delta)\equiv I(R(x'),\Delta)\bmod 2$ and $P(R(x),\Delta)=P(R(x'),\Delta)$.
\end{proof}
By repeatedly using the above lemma, we get the following corollary.
\begin{cor}
 Let $P$ be a polygonal path such that $P\cap\Delta=\emptyset$ and no ray from any point of $P$ hits a vertex of $\Delta$. Then all points of $P$ have the same parity.\qed
\end{cor}
\begin{cor}
Let $x$ be a point from $\mathbb{R}^3$ such that $\Delta\cap x=\emptyset$ and $R(x)$ hits a vertex of $\Delta$. Then there is a neighborhood $U(x)$ of $x$ such that all points from $U(x)\setminus x$ have the same parity.
\end{cor}
\begin{proof}
Let $U(x)$ be a neighborhood of $x$ such that $R(y)$ does not hit a vertex of $\Delta$ for $y\in U(x)\setminus x$. We can connect any two points of $U(x)\setminus x$ by a polygonal path and use the previous corollary.
\end{proof}

Let $x$ be a point of $\mathbb{R}^3$ such that $R(x)$ intersect a vertex of $\Delta$. We define the parity of $x$ to be the same as a parity of a sufficiently small neighborhood of $x$.

\begin{cor}
 Let $P$ be a polygonal path such that $P\cap\Delta=\emptyset$. Then all points of $P$ have the same parity.\qed
\end{cor}

\begin{proof}[Finish of the proof of Proposition~\ref{prop:divides}]
Any two points of a connected region of triangular configuration can be connected by a polygonal path. Hence any two points of a connected region have the same parity.

Let $a$ and $b$ be two different points of $\mathbb{R}^3$ such that $a$ and $b$ lie close to a triangle $t$ of $\Delta$ and the segment from $a$ to $b$ intersects $\Delta$ only on the interior of $t$. Then $a$ and $b$ have different parities. Hence, $\Delta$ has at least two cells.
\end{proof}
\subsubsection{Proof of Theorem~\ref{thm:embhas2bas}}
\label{sec:proofembhas2bas}
\begin{prop}
\label{prop:independent}
Let $\Delta$ be a triangular configuration embedded into $\mathbb{R}^3$. Then the set $\mathcal{S}$ of characteristics vectors of elementary strong boundaries of $\Delta$ is linear independent.
\end{prop}
\begin{proof}
We prove the proposition by the induction along the size of the set $\mathcal{S}$.
If $\lvert\mathcal{S}\rvert\leq 1$, the proposition is clear.
Let $\lvert\mathcal{S}\rvert>1$ and let $\chi(C)$ be an element of $\mathcal{S}$ and let $C$ be the corresponding elementary strong boundary such that $C$ is incident with the unbounded cell of $\Delta$.
Let $t$ be a triangle of $C$ that is incident with the unbounded cell.
Now we show that the triangle $t$ belongs only to one elementary strong boundary $C$. 
For a contradiction suppose that $t$ belongs to an elementary strong boundary $C'$ of $\Delta$ different from $C$. Let $X_1$ and $X_2$ be the cell of $\Delta$ incident with $t$. One of the cells is unbounded, suppose that $X_2$ is unbounded.
Then $X_2\subseteq\inter{C}$ and $X_2\subseteq\inter{C'}$. Thus, $\inter{C}\cap\inter{C'}\neq\emptyset$. Since $C\neq C'$, $\inter{C'}\cap\ext{C}\neq\emptyset$. Hence, $C'$ is not elementary strong boundary. The contradiction. Thus, $t$ belongs to only one elementary strong boundary $C$.

Hence, $\chi(C)$ is not linear combination of the other elements $\mathcal{S}\setminus\{\chi(C)\}$.
By the induction assumption, the set $\mathcal{S}\setminus\{\chi(C)\}$ is linear independent. Hence, the set $\mathcal{S}$ is linear independent.
\end{proof}

\begin{thm}
\label{thm:s2basiscon}
Let $\Delta$ be a connected triangular configuration embedded into $\mathbb{R}^3$.
Let $\mathcal{S}$ be the set of characteristics vectors of elementary strong boundaries of $\Delta$. Then the set $\mathcal{S}$ is a 2-basis of the cycle space $\ker\Delta$ of $\Delta$.
\end{thm}
\begin{proof}
Let $\chi(C_0)$ be an element of $\ker\Delta$ and let $E_0$ be the subset of triangles of $\Delta$ such that $C_0=\mathcal{K}(E_0)$. The set $E_0$ is an even subset of $\Delta$.

By Proposition~\ref{prop:divides}, the triangular configuration $\mathcal{K}(E_0)$ has at least two cells.
Therefore, $\mathcal{K}(E_0)$ contains a strong boundary $C_1$. Let $E_1$ be the subset of triangles of $\Delta$ such that $C_1=\mathcal{K}(E_1)$.
By Proposition~\ref{prop:boundeven}, the set $E_1$ is an even subset.
The symmetric difference $E_0\bigtriangleup E_1$ is also even subset.
For $i=2,\dots,k$ we define the sets $E_i$ in the following way:
Until $E_{0}\bigtriangleup\dots\bigtriangleup E_{i-1}\neq\emptyset$ we set $E_i$ to be the set of triangles of a strong boundary contained in $\mathcal{K}(E_{0}\bigtriangleup\dots\bigtriangleup E_{i-1})$.
The triangular configuration $\mathcal{K}(E_{0}\bigtriangleup\dots\bigtriangleup E_{i-1})$ contains a strong boundary, because $E_{0}\bigtriangleup\dots\bigtriangleup E_{i-1}$ is even subset and by Proposition~\ref{prop:divides}, triangular configuration $\mathcal{K}(E_{0}\bigtriangleup\dots\bigtriangleup E_{i-1})$ has at least two cells.
Since $\Delta$ is finite, this sequence of even subsets is finite.

Thus, the set $E_0$ is the symmetric difference of the even subsets $E_1,\dots,E_k$ and $\chi(C_0)=\chi(C_1)+\dots+\chi(C_k)$ over $GF(2)$.
By proposition~\ref{prop:boundcomb}, characteristics vector of each strong boundary $\chi(C_i)$, $i=1,\dots,k$; is a linear combination of characteristics vectors of elementary strong boundaries over $GF(2)$. Therefore, $\chi(C)$ is a linear combination of characteristics vectors of elementary strong boundaries $\mathcal{S}$. By Proposition~\ref{prop:independent}, the set $\mathcal{S}$ is linear independent. Thus, the set $\mathcal{S}$ is a basis.

Every strong boundary has exactly two cells.
By definition of elementary strong boundary, the inner cell of any elementary strong boundary contains no triangle of other strong boundary.
Hence, every triangle of $\Delta$ is contained in at most two elementary strong boundaries and at most two characteristics vectors of elementary strong boundaries are non-zero on the same coordinate.
Thus, the set $\mathcal{S}$ is a 2-basis.
\end{proof}

\begin{thm}
\label{thm:s2basis}
Let $\Delta$ be a triangular configuration embedded into $\mathbb{R}^3$. Then the cycle space of $\Delta$ has a 2-basis.
\end{thm}
\begin{proof}
Let $\Delta_1,\dots,\Delta_m$ be connected components of the triangular configuration $\Delta$. Let $B_1,\dots,B_m$ be bases of $\Delta_1,\dots,\Delta_m$, respectively, provided by Theorem~\ref{thm:s2basiscon}. Since characteristics vectors that corresponds to strong boundaries from different connected components have no common non-zero coordinate, the set $B:=B_1\cup\dots\cup B_m$ is a 2-basis of the cycle space of $\Delta$.
\end{proof}


\subsection{Proof of Theorem \ref{thm:repr3d} (Representations in $\mathbb{R}^3$)}
\label{sec:repr3d}

It remains to prove sufficiency of the condition of Theorem~\ref{thm:repr3d} for geometric representations in $\mathbb{R}^3$.
We show that the construction from Rytíř~\cite{rytir2,rytir3} for binary linear codes with 2-basis can be embedded into $\mathbb{R}^3$.
\begin{figure}
 \centering
 \includegraphics{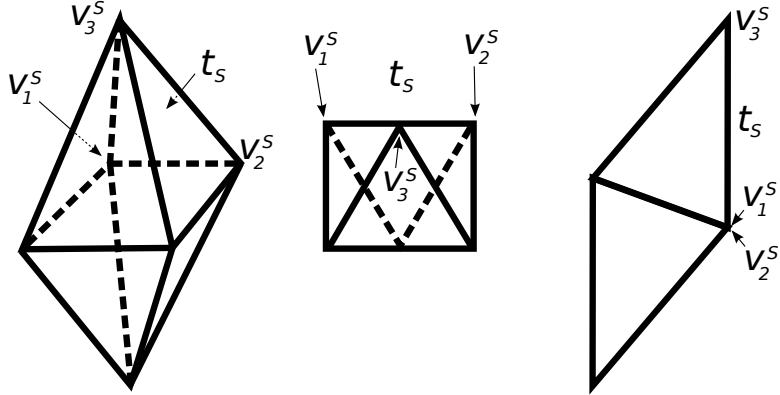}
 \caption{Eight triangles forming triangular sphere $S$. The picture on the left is a perspective view, the middle picture is a view from top, the picture on the right is a view from the right side.}
 \label{fig:basicsphere}
\end{figure}
\subsubsection{Basic building blocks}
We start with definition of basic building blocks.

\subsubsection{Triangular configuration $S^n$}
\label{sec:sn}

First, we define triangular configuration $S$ as a triangulation of a two dimensional sphere by $8$ triangles. It is depicted in Figure~\ref{fig:basicsphere}. The triangle $t_S$ has vertices $v^S_1,v^S_2,v^S_3$. All triangles of $S$ have the same size. Therefore, the size of $S$ and position of $S$ in a space is determined by the coordinates of the points $v^S_1,v^S_2,v^S_3$. We denote the triangular configuration $S$ with prescribed vertices $v^S_1=x,v^S_2=y,v^S_3=z$ by $S(x,y,z)$.
\begin{prop}
Triangular configuration $S$ can be embedded into $\mathbb{R}^3$.\qed
\end{prop}

Let $n$ be a positive integer. We subdivide the triangle $t_S$ of $S$ in the way depicted in figure~\ref{fig:extspheresub}. Note that, the resulting object is a triangular configuration. We denote the resulting triangular configuration by $S^n$. Clearly, $S^n$ can be embedded into $\mathbb{R}^3$.
We denote the triangle $i$ of $S^n$ by $S^n(i)$, for $i=1,\dots,n$.

\begin{figure}
 \centering
 \includegraphics{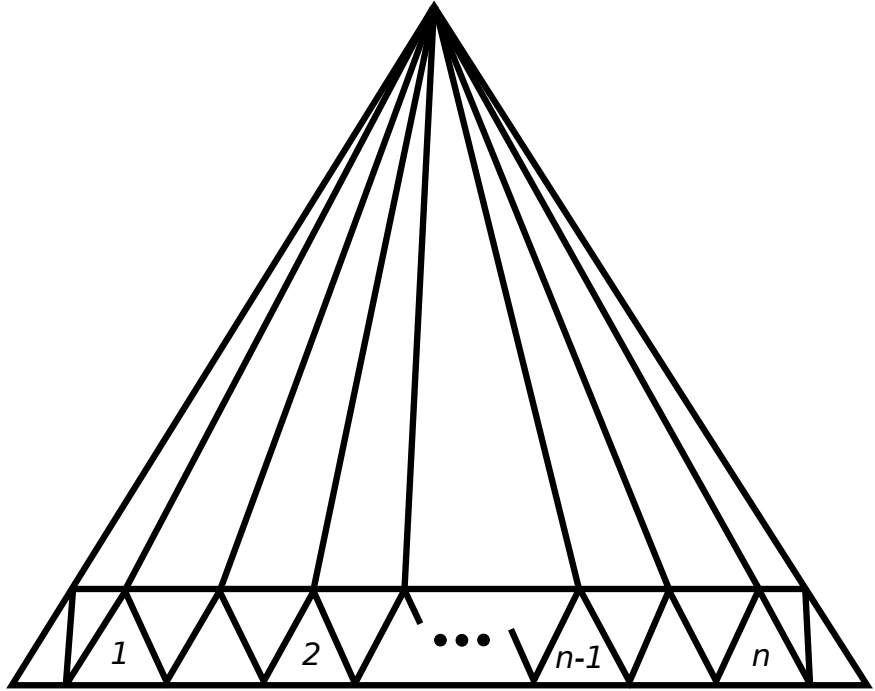}
 \caption{Subdivision of a triangle, triangles $1,\dots,n$ are equilateral.}
 \label{fig:extspheresub}
\end{figure}

\subsubsection{Triangular tunnel}
\label{sec:tunnel}
Let $t_1$ and $t_2$ be two empty triangles. Let $x_1,x_2,x_3$ be vertices of $t_1$ and $y_1,y_2,y_3$ be vertices of $t_2$.
The triangular tunnel between $t_1$ and $t_2$ denoted by $T(t_1,t_2)$ is the six triangles that form a tunnel as is depicted in Figure~\ref{fig:tunnel}. The vertices of the empty triangle $abc$ lies on the points $x_1,x_2,x_3$ and the vertices of the empty triangle $123$ lies on the points $y_1,y_2,y_3$.

\begin{figure}
 \centering
 \includegraphics{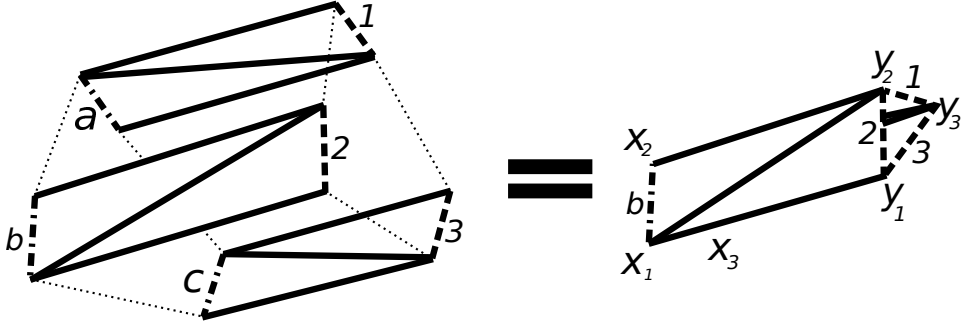}
 \caption{Triangular tunnel $T(t_1,t_2)$}
 \label{fig:tunnel}
\end{figure}

\subsubsection{Triangular tunnel bridge}
Let $t_1$ and $t_2$ be empty triangles embedded into $\mathbb{R}^3$ such that $t_1$ and $t_2$ belong to the hyperplane given by equation $x_3=0$ and one edge of both $t_1$ and $t_2$ belongs to $x_1$ axis of $\mathbb{R}^3$ and $t_2$ is a shifted copy of $t_1$ in the direction of $x_1$ axis of $\mathbb{R}^3$, $t_2=t_1+a(1,0,0)$, $a\in\mathbb{R}$. Let $l$ be the size of edge of $t_1$. We suppose that $a$ is greater than $l$. See Figure~\ref{fig:bridge1}.
\begin{figure}
 \centering
 \includegraphics{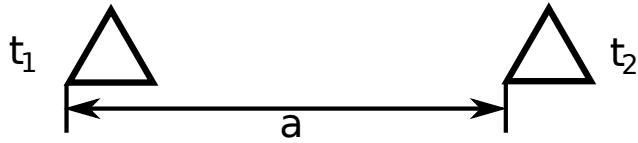}
 \caption{Empty triangles of bridge}
 \label{fig:bridge1}
\end{figure}

Let $b>a$ and $c>a$. Let $alt(t_1)$ and $alt(t_2)$ denote the altitude of $t_1$ and $t_2$, respectively. Let $t'_1$ and $t'_2$ be copies of triangle $t_1$ and $t_2$ shifted by $(0,b,0)$ with top vertex shifted by $(0,-l/2,0)$, respectively.
Let $t''_1$ be a copy of $t'_1$ shifted by $(0,0,c)$ with the left vertex shifted by $(0,0,alt(t_1))$ and let $t''_2$ be a copy of $t'_2$ shifted by $(0,0,c)$ with the right vertex shifted by $(0,0,alt(t_2))$. Then the triangular tunnel bridge is $$TB(t_1,t_2,b,c):=T(t_1,t'_1)\cup T(t'_1,t''_1)\cup T(t''_1,t''_2)\cup T(t''_2,t'_2)\cup T(t'_2,t_2).$$
The triangular tunnel bridge is depicted in Figure~\ref{fig:bridge}.

\begin{figure}
 \centering
 \includegraphics{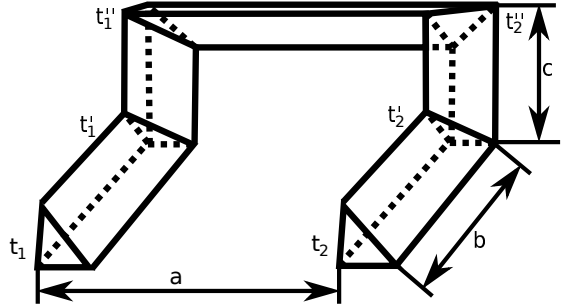}
 \caption{Triangular tunnel bridge}
 \label{fig:bridge}
\end{figure}

\begin{prop}
\label{prop:2bridges}
Let $t_1,t_2,t_3,t_4$ be disjoint triangles embedded into $\mathbb{R}^3$ such that the triangles belong to the hyperplane given by equation $x_3=0$ and one edge of each $t_1,t_2,t_3,t_4$ belongs to $x_1$ axis of $\mathbb{R}^3$ and $t_2,t_3,t_4$ are shifted copies of $t_1$ in the direction of the first coordinate of $\mathbb{R}^3$.
Let $l$ be the size of the longest edge of $t_1$. Let $a>l$ and $b>2a$.
Then the triangular tunnel bridges $TB(t_1,t_2,a,a)$ and $TB(t_3,t_4,b,a)$ are disjoint.
\end{prop}
\begin{proof}
The proposition follows from Figure~\ref{fig:2bridges}.
\begin{figure}
 \centering
 \includegraphics{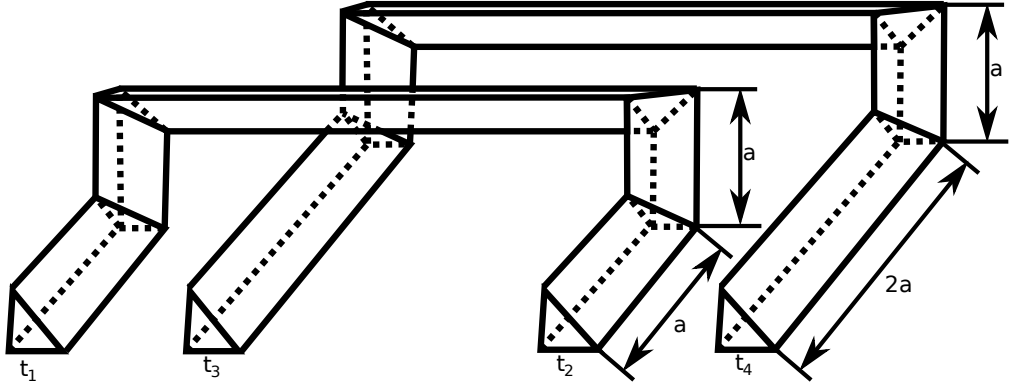}
 \caption{The proof of Proposition~\ref{prop:2bridges}.}
 \label{fig:2bridges}
\end{figure}

\end{proof}

\subsubsection{Construction}
Let $\mathcal{C}$ be a binary code with a 2-basis $B=\{b_1,\dots,b_d\}$. We construct the following triangular configuration $\Delta^\mathcal{C}_B$ and embed it into $\mathbb{R}^3$.

In the first step, we put $d$ identical copies of $S^n$, denoted by $S^n_1,\dots,S^n_d$; into $\mathbb{R}^3$ as is depicted in Figure~\ref{fig:spheres}.
\begin{figure}
 \centering
 \includegraphics{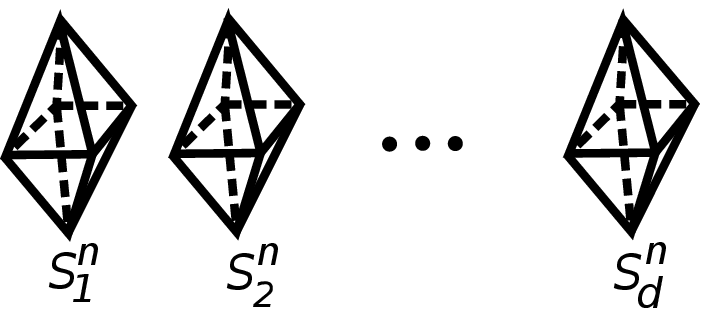}
 \caption{Spheres}
 \label{fig:spheres}
\end{figure}
Formally:
Let $v^1_1$ equals $(0,0,0)$ and $v^1_2$ equals $(2,0,0)$ and $v^1_3$ equals $(1,0,1)$.
The points $v^1_1,v^1_2,v^1_3$ are vertices of the first copy of $S^n$. Thus $S^n_1$ equals $S^n(v^1_1,v^1_2,v^1_3)$. Therefore, the size of every edge of every triangle of $S^n_1$ is less or equal $2$. 
The triangular configuration $S^n_i$ is shifted by offset $5i$ from the origin.
Let $v^i_1$ equals $(0+5i,0,0)$ and $v^i_2$ equals $(2+5i,0,0)$ and $v^i_3$ equals $(1+5i,0,1)$.
Then $S^n_i$ equals $S^n(v^i_1,v^i_2,v^i_3)$, for $i=1,\dots,d$.
Then $$\Delta^\mathcal{C}_B:=S_1^n\cup\dots\cup S^d_n.$$
In the second step, we add to $\Delta^\mathcal{C}_B$ the tunnels.
We also construct set of triangles $\{B_1^n,\dots,B_n^n\}$ and triangular configurations $\Delta_{b_i}, i=1,\dots,d$. Initially we set $\Delta_{b_i}:=S^n_i$ for $i=1,\dots,d$.
We interconnect the triangular configurations $S^n_i$, for $i=1,\dots,d$, by tunnels in the following way:

We proceed from the first coordinate $1$ to the last coordinate $n$.
\begin{itemize}

\item If the coordinate $i$ is zero in all basis vectors, we add an isolated triangle to $\Delta^\mathcal{C}_B$ and denote it by $B^n_i$.

\item If the coordinate $i$ is non-zero only in one basis vector $b_k$, we denote the triangle $S^n_k(i)$ by $B^n_i$ and we do nothing otherwise.

\item If the coordinate $i$ is non-zero in two basis vectors $b_k$ and $b_l$, $k<l$, we add triangular tunnel bridge $TB(S^n_k(i),S^n_l(i),5i,5)$ to $\Delta^\mathcal{C}_B$.
We also add this tunnel bridge to the set $\Delta_{b_l}$.
We remove the triangle $S^n_l(i)$ from $\Delta^\mathcal{C}_B$ and $\Delta_{b_l}$ and we denote the triangle $S^n_k(i)$ by $B^n_i$.

\end{itemize}

We denote the set of triangles $\{B^n_1,\dots,B^n_n\}$ by $B^n$.

\begin{prop}
The triangular tunnel bridges added in the last step are mutually disjoint.
\end{prop}
\begin{proof}
Let $TB(S^n_{k_1}(i_1),S^n_{l_1}(i_1),5i_1,5)$ and $TB(S^n_{k_2}(i_2),S^n_{l_2}(i_2),5i_2,5)$ be two triangular tunnel bridges from the last step. If there is none or only one, the proposition follows.
Since $i_1\neq i_2$ and $k_1\neq l_1$ and $k_2\neq l_2$,
the triangles $S^n_{k_1}(i_1),S^n_{l_1}(i_1),S^n_{k_2}(i_2),S^n_{l_2}(i_2)$ are disjoint. The size of each edge of the triangles is at most $2$.
Since $i_1\geq 1, i_2\geq 1$, it holds $5i_1>2$ and $5i_2>2$. We can suppose that $i_1<i_2$. Therefore $2(5i_1)\leq 5i_2$. Now, we can use Proposition~\ref{prop:2bridges} and the proposition follows.
\end{proof}

\begin{cor}
Triangular configuration $\Delta^\mathcal{C}_B$ can be embedded into $\mathbb{R}^3$.\qed
\end{cor}

An example of construction is depicted in Figure~\ref{fig:exampleconstr}.
\begin{figure}
 \centering
 \includegraphics{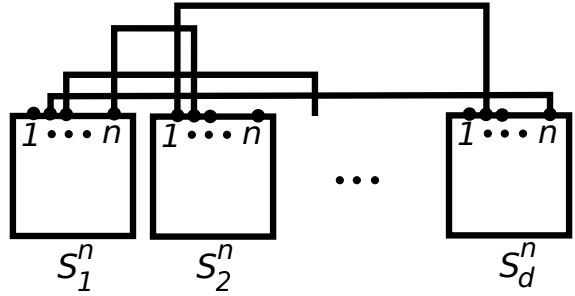}
 \caption{Top view on an example of the construction, the triangular tunnel bridges are depicted by lines connected to dots denoted by $1,\dots,n$.}
 \label{fig:exampleconstr}
\end{figure}
To finish proof of Theorem~\ref{thm:repr3d}, it remains to show that $\Delta^\mathcal{C}_B$ is geometric representation of $\mathcal{C}$.
We prove that $\Delta^\mathcal{C}_B$ is indeed geometric representation of $\mathcal{C}$ in Subsection~\ref{sec:represents}.

\subsubsection{Proof of representability}
\label{sec:represents}
We follow strategy described in Rytíř~\cite{rytir2} with the building blocks constructed in previous section. Before we state the proofs we introduce some definitions. In this section all operations are over the field $GF(2)$.

Let $\mathcal{C}$ be a binary linear code and let $B=\{b_1,\dots,b_d\}$ be a basis of $\mathcal{C}$.
Let $\Delta^\mathcal{C}_B$ be the geometric representation of $\mathcal{C}$ with respect to the basis $B$ from Section~\ref{sec:repr3d} or Section~\ref{sec:repr4d}.
We suppose that $\Delta^\mathcal{C}_B$ exists.
Let $c$ be a codeword from $\mathcal{C}$.
Then $c=\sum_{i\in I}b_i$.
The \dfn{degree} of $c$ with respect to the basis $B$ is defined to be the cardinality $\left\rvert I\right\lvert$ of the index set.
The degree is denoted by $\vd{c}$.
 Let $\Delta^\mathcal{C}_{b_i}$, $i=1,\dots,d$ be triangular configurations defined also in Section~\ref{sec:repr3d}.
We define a linear mapping $f\colon\mathcal{C}\mapsto \ker\Delta^\mathcal{C}_B$ in the following way:
Let $c$ be a codeword of $\mathcal{C}$ and let $c=\sum_{i\in I}b_i$ be the unique expression of $c$, where $b_i\in B$.
We define $f(c):=\sum_{i\in I}\chi(\Delta^\mathcal{C}_{b_i})$.
The entries of $f(c)$ are indexed by the triangles of $\Delta^\mathcal{C}_B$.
We have $f(c)^{B^n_j}=1$ if and only if $\bigtriangleup_{i\in I} T(\Delta^\mathcal{C}_{b_i})$ contains the triangle $B^n_j$.

\begin{prop}
\label{prop:mappingf}
Let $m$ be the number of triangles of $\Delta^\mathcal{C}_B$.
Let $c=(c^1,\dots,c^n)$ and
\begin{equation*}
f(c)=\left(f(c)^{B^n_1},\dots,f(c)^{B^n_n},f(c)^{n+1},\dots,f(c)^{m}\right). 
\end{equation*}
Then $f(c)^{B^n_j}=c^j$ for all $j=1,\dots,n$ and all $c\in\mathcal{C}$.
\end{prop}

\begin{proof}

We show the proposition by the induction on the degree $\vd{c}$ of $c$.
The codeword $c$ is equal to $\sum_{i\in I}b_i$.
If $\vd{c}=0$, then $c=0$ and $f(c)=0$. Thus, $f(c)$ is the characteristics vector of the empty triangular configuration. The proposition holds for vectors of degree $0$.
Suppose that $\vd{c}$ is greater than $0$, then $\lvert I\rvert\geq 1$. We choose some $k$ from $I$. The codeword $c+b_k$ has a degree less than $c$. By the induction assumption, the proposition holds for $c+b_k$.
Let $b_k=(b_k^1,\dots,b_k^n)$.
From the definition of $\Delta^\mathcal{C}_{b_k}$, the equality $b_k^j=\chi(\Delta^\mathcal{C}_{b_k})^{B^n_j}$ holds for all $j=1,\dots,n$.
Therefore,
\begin{equation*}
c^j=(c^j+b_k^j)+b_k^j=\chi(\bigtriangleup_{i\in I\setminus\{k\}}\Delta^\mathcal{C}_{b_i})^{B^n_j}+\chi(\Delta^\mathcal{C}_{b_k})^{B^n_j}=f(c)^{B^n_j}
\end{equation*}
for all $j=1,\dots,n$.
\end{proof}
\begin{cor}
\label{cor:inj}
The mapping $f$ is injective.\qed
\end{cor}

\begin{lem}
\label{lem:cycletriangles}
Let $E$ be a non-empty even subset of $\Delta^\mathcal{C}_B$.
Then $\mathcal{K}(E)$ contains $\Delta^\mathcal{C}_{b_i}\setminus T(B^n)$ ($\Delta^\mathcal{C}_{b_i}$ with triangles of $B^n$ removed) as a subconfiguration for some $i\in \left\{1,\dots,d\right\}$.
\end{lem}
\begin{proof}
Triangular configuration $\mathcal{K}(E)$ contains either all triangles or no triangle of $\Delta^\mathcal{C}_{b_i}\setminus T(B^n)$, since all edges of $\Delta^\mathcal{C}_{b_i}$ incident with no triangle of $B^n$ have degree equals $2$, for $i=1,\dots,d$.
The triangular configuration $B^n$ have no non-empty cycle, since the triangles of $B^n$ are disjoint.
Hence, $\mathcal{K}(E)$ contains a triangle of $\Delta^\mathcal{C}_{b_i}\setminus T(B^n)$ for some $i\in \left\{1,\dots,d\right\}$. Thus, $\mathcal{K}(E)$ contains $\Delta^\mathcal{C}_{b_i}\setminus T(B^n)$ for some $i\in \left\{1,\dots,d\right\}$. 
\end{proof}

\begin{thm}
\label{thm:trianrep}
The mapping $f$ defined above is a bijection between the binary linear code
$\mathcal{C}$ and $\ker\Delta^\mathcal{C}_B$.
\end{thm}
\begin{proof}
By Corollary~\ref{cor:inj}, the mapping $f$ is injective.
It remains to be proven that $\dim\mathcal{C}=\dim\ker\Delta^\mathcal{C}_B$.
Suppose on the contrary that some codeword of $\ker\Delta^\mathcal{C}_B$ is not in the span of $\left\{f(b_1),\dots,f(b_d)\right\}$.
Let $c$ be such a codeword with the minimal possible weight $w(c)$.
The weight $w(c)$ means the number of non-zero coordinates of $c$. 
Let $E$ be an even subset of $\Delta^\mathcal{C}_B$ such that $\chi(\mathcal{K}(E))=c$.
By Lemma~\ref{lem:cycletriangles}, $\mathcal{K}(E)$ contains $\Delta^\mathcal{C}_{b_i}\setminus T(B^n)$ for some $i\in \left\{1,\dots,d\right\}$.
By definition of $\Delta^\mathcal{C}_{b_i}$, it holds $\left\lvert T(\Delta^\mathcal{C}_{b_i}\setminus T(B^n))\right\rvert>\left\lvert T(B^n)\right\rvert$.
Therefore, the inequality $\left\lvert E\bigtriangleup T(\Delta^\mathcal{C}_{b_i})\right\rvert < \left\lvert E\right\rvert$ holds.
Thus, $w(c)>w(\chi(\mathcal{K}(E\bigtriangleup T(\Delta^\mathcal{C}_{b_i})))$. This is a contradiction.

\end{proof}

The entries of the vectors of $\ker\Delta^\mathcal{C}_B$ are indexed by triangles and the entries of vectors of $\mathcal{C}$ are indexed by integers, we make a convention that a coordinate of $\ker\Delta^\mathcal{C}_B$ indexed by triangle $B^n_i$ corresponds to coordinate of $\mathcal{C}$ indexed by $i$. Now, we can state the following corollary.
\begin{cor}
$\mathcal{C}=\ker\Delta^\mathcal{C}_B/(T(\Delta^\mathcal{C}_B\setminus T(B^n)))$ and $\dim\ker\Delta^\mathcal{C}_B=\dim\mathcal{C}$.\qed
\end{cor}
Thus, the triangular configuration $\Delta^\mathcal{C}_B$ is a geometric representation of $\mathcal{C}$.

\subsection{Proof of Theorem~\ref{thm:repr4d} (Every binary code has a representation in $\mathbb{R}^4)$}
\label{sec:repr4d}
In this section for every binary linear code we construct its geometric representation that can be embedded into $\mathbb{R}^4$.

Let $\mathcal{C}$ be a binary linear code of length $n$ and let $B=\{b_1,\dots,b_d\}$ be a basis of $\mathcal{C}$.
For every basis vector $b_i$ we construct triangular configuration $\Delta_{b_i}$ in this way: Let $Q$ be a three dimensional cube of size $1\times1\times1$. We put in the middle of this cube the triangular configuration $S^n$ defined in Section~\ref{sec:sn}.
We make an appropriate scaling of $S^n$ such that $S^n$ fits into the cube $Q$ and put $S^n$ into $Q$ in the way depicted in Figure~\ref{fig:snembededcube}. The triangles $S^n(k)$, $k=1,\dots,n$ of $S^n$ are in front.
\begin{figure}
 \centering
 \includegraphics{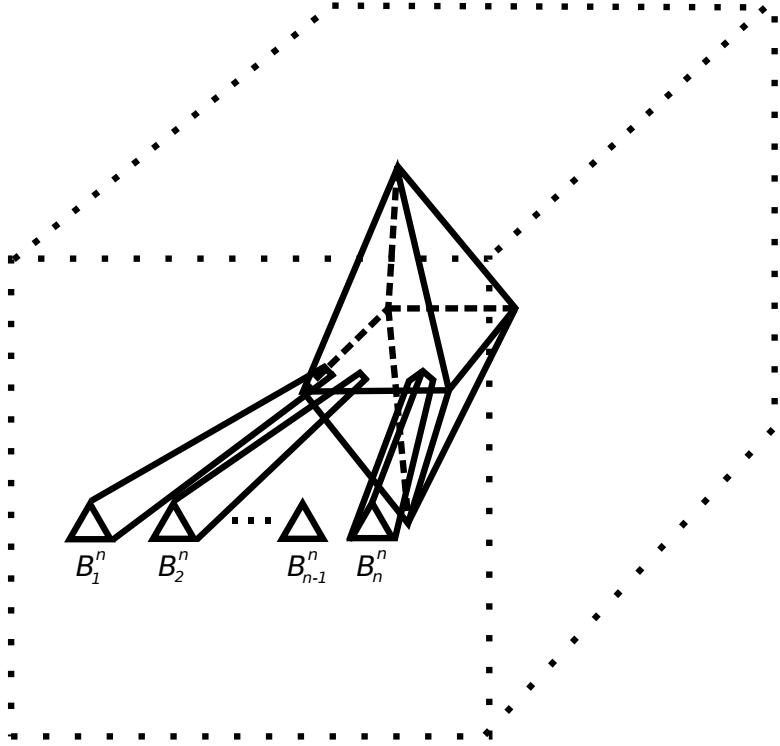}
 \caption{An example of $\Delta_{b_i}$ put into cube $Q$ for $b_i=(1,1,\dots,0,1)$.}
 \label{fig:snembededcube}
\end{figure}
Let $F$ be the front facet of $Q$ (front in the Figure~\ref{fig:snembededcube}). We put triangles $\{B^n_1,\dots,B^n_n\}$ to $F$ as is depicted in Figure~\ref{fig:snembededcube}.
Let $b_i$ equals $(b_i^1,\dots,b_i^n)$.
We initially set $\Delta_{b_i}:=S^n$.
For every non-zero coordinate $b_i^k$ we add tunnel (See Section~\ref{sec:tunnel}) $T(S^n(k),B^n_k)$ between triangles $S^n(k)$ and $B^n_k$ to $\Delta_{b_i}$.
Then we remove triangle $S^n(k)$ from $\Delta_{b_i}$.
An example of $\Delta_{b_i}$ for $b_i=(1,1,\dots,0,1)$ is depicted in Figure~\ref{fig:snembededcube}.
The cube $Q$ is not a part of $\Delta_{b_i}$, it is important that $\Delta_{b_i}$ is embedded into $Q$. We denote this cube by $Q(\Delta{b_i})$ and the facet $F$ by $F(\Delta_{b_i})$.

\begin{prop}
\label{prop:cubes4d}
Let $Q_1,\dots,Q_d$ be three dimensional cubes of the same size. Then the cubes can be embedded into $\mathbb{R}^4$ such that all cubes intersect at one facet and otherwise are disjoint.
\end{prop}
\begin{proof}
Fix a size $l$ of the edges of the cubes. Let $F$ be a square of size $l\times l$ embedded into $\mathbb{R}^4$. Let $v_1,\dots,v_d$ be vectors of $\mathbb{R}^4$ of length $l$ orthogonal to the square $F$ such that every two vectors of $v_1,\dots,v_d$ are linear independent.
Such vectors exist in $\mathbb{R}^4$.
Let $Q_i$ be the cube defined as $\{f+\alpha v_i\vert f\in F , \alpha\in[0,1]\}$.
The cubes intersect at facet $F$:
for a contradiction suppose that there are two cubes $Q_i$ and $Q_j$ such that $Q_i\cap Q_j\nsubseteq F$. Let $x$ be a point of $(Q_i\cap Q_j)\setminus F$.
Then $x=f_i+\alpha_i v_i=f_j+\alpha_j v_j$, where $f_i,f_j\in F$ and $\alpha_i,\alpha_j\in(0,1]$.
Since $v_i$ is not a linear combination of $v_j$, the points $f_i,f_j$ are different.
The point $f_i+\alpha_i v_i-\alpha_j v_j$ belongs to $F$. Thus, the vector $\alpha_i v_i-\alpha_j v_j$ is parallel to $F$.
Since $f_i,f_j$ are different, we have $\alpha_i v_i-\alpha_j v_j\neq0$.
Since the vector $\alpha_i v_i-\alpha_j v_j$ is a linear combination of two vectors $v_i,v_j$ orthogonal to $F$, the vector $\alpha_i v_i-\alpha_j v_j$ is also orthogonal to $F$. Thus, the vector $\alpha_i v_i-\alpha_j v_j$ is non-zero and orthogonal to itself. This is impossible in $\mathbb{R}^4$, a contradiction. The proposition follows.
\end{proof}

By Proposition~\ref{prop:cubes4d}, we can embed the cubes $Q(\Delta{b_1}),\dots,Q(\Delta{b_d})$ with $\Delta_{b_1},\dots,\Delta_{b_d}$ into $\mathbb{R}^4$ such that all cubes $Q(\Delta{b_1}),\dots,Q(\Delta{b_d})$ intersect on the facets $F(\Delta_{b_1}),\dots,F(\Delta_{b_d})$ and otherwise are disjoint and the triangular configurations $\Delta_{b_1},\dots,\Delta_{b_d}$ intersects on the triangles $B^n$ and otherwise are disjoint. The resulting triangular configuration is the geometric representation $\Delta^\mathcal{C}_B$ of $\mathcal{C}$ embedded into $\mathbb{R}^4$. An example of the representation of a binary linear code that is generated by two basis vectors $b_1,b_2$ is depicted in Figure~\ref{fig:constr4b}.

\begin{figure}
 \centering
 \includegraphics{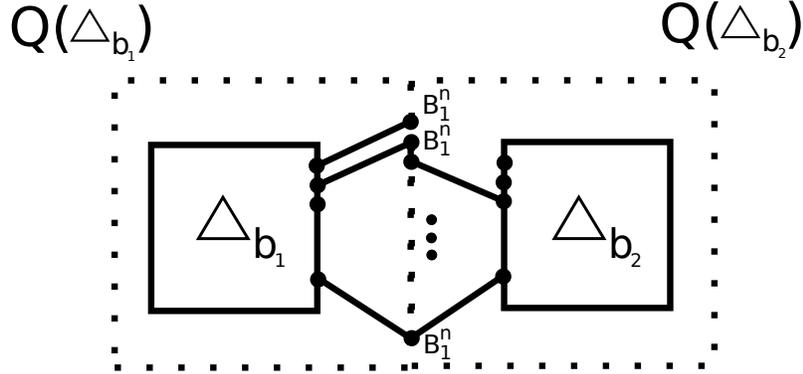}
 \caption{Top view on an example of the representation of a two dimensional code in $\mathbb{R}^4$}
 \label{fig:constr4b}
\end{figure}

The proof that $\Delta^\mathcal{C}_B$ is indeed geometric representation of $\mathcal{C}$ is the same as the proof in Subsection~\ref{sec:represents}.
This completes the proof of Theorem~\ref{thm:repr4d}.

\subsection{Proof of Theorem~\ref{thm:repr3dgraphs}}
\begin{prop}
\label{prop:2basisgraphs}
Let $\mathcal{C}$ be a binary linear code.
Then $\mathcal{C}$ has a 2-basis if and only if there is a graph $G$ such that $\mathcal{C}$ is equal to the cut space of $G$.
\end{prop}
\begin{proof}
First, we prove that every binary linear code with a 2-basis is a cut space of a graph possibly with loops and parallel edges.
Let $\mathcal{C}$ be a binary linear code of length $n$ with a 2-basis $B=\{b_1,\dots,b_d\}$.
We define a graph $G=(V,E)$ possibly with parallel edges and loops as follows:
We define the set of vertices $V$ as:
$$V:=B\cup\{u\}.$$
For $i=1,\dots,n$; we define edge $e_i$ as follows:
If all basis codewords of $B$ have the entry indexed by coordinate $i$ equals to zero,
we set $e_i$ to be a loop $(u,u)$.
If there is exactly one basis codeword $b_l\in B$ that has non-zero entry indexed by $i$, we set $e_i$ to be $(b_l,v)$.
If there are exactly two basis codewords $b_l,b_k\in B$ that have non-zero entry indexed by $i$, we set $e_i$ to be $(b_l,b_k)$.
Then the set of edges $E$ of $G$ is $\{e_i\vert i=1,\dots,n\}$.

Let $E(v)$ be the set of edges incident with a vertex $v$. Let $E'$ be a subset of $E$. We define the incidence vector of $E'$ is $\chi(E'):=(\chi(E')_1,\dots,\chi(E')_n)$, where $\chi(E')_i=1$ if $e_i\in E'$ and $\chi(E')_i=0$ otherwise.
By definition, the set $B':=\{\chi(E(v))\vert v\in (V\setminus \{u\})\}$ equals $B$. It is known fact that the set $B'$ generates the cut space of $G$, for a proof see for example Diestel~\cite{diestel}.

Now, we prove the reverse implication. Let $G$ be a graph and let $u$ be a vertex of $G$. Then the set $B':=\{\chi(E(v))\vert v\in (V\setminus \{u\})\}$ is a basis of the cut space of $G$. Since every edge of $G$ is incident at most with two vertices, the set $B'$ is a 2-basis of the cut space of $G$.
\end{proof}

Now, we finish the proof of Theorem~\ref{thm:repr3dgraphs}.
\begin{proof}[Proof of Theorem~\ref{thm:repr3dgraphs}]
By proposition~\ref{prop:2basisgraphs}, a binary linear code $\mathcal{C}$ has 2-basis if and only if $\mathcal{C}$ is a cut space of a graph.
By Theorem~\ref{thm:repr3d}, code $\mathcal{C}$ has representation in $\mathbb{R}^3$ if and only if it has a 2-basis. If code $\mathcal{C}$ has no 2-basis. Then from Theorem~\ref{thm:repr4d} follows that $\mathcal{C}$ has representation in $\mathbb{R}^4$.
\end{proof}
\begin{proof}[Proof of Corollary~\ref{cor:dec3d}]
It is a known fact that there is polynomial algorithm that decide if a given binary linear code is a cut space of a graph (See Seymour~\cite{seymourgraphic}).
\end{proof}

\begin{prop}
Let $G$ be a non-planar graph. Then the cycle space of $G$ has no 2-basis.
\end{prop}
\begin{proof}
Follows from the Mac Lane's planarity criterion. See Mac Lane~\cite{maclane} or O'Neil~\cite{oneil}.
\end{proof}

\section*{Acknowledgment}
I would like to thank Martin Loebl for helpful discussions and continual support.

\bibliography{references}{}
\bibliographystyle{abbrv}

\end{document}